\documentclass[12pt]{amsart}

\usepackage[english]{babel}
\usepackage{mathrsfs,amssymb}
\usepackage{mathtools}
\usepackage[colorlinks, citecolor = blue]{hyperref}

\usepackage[shortlabels]{enumitem}
\setlist[itemize]{leftmargin=25pt}
\setlist[enumerate]{leftmargin=25pt}

\newtheorem{theorem}{Theorem}[section]
\newtheorem{lemma}[theorem]{Lemma}
\newtheorem{prop}[theorem]{Proposition}
\newtheorem{cor}[theorem]{Corollary}

\theoremstyle{definition}

\newtheorem{con}[theorem]{Conjecture}

\theoremstyle{remark}
\newtheorem{remark}[theorem]{Remark}

\numberwithin{equation}{section}
\usepackage[colorinlistoftodos,prependcaption,textsize=small]{todonotes}

\let \la=\lambda

\let \d=\delta

\let \a=\alpha
\let \f=\varphi

\let \ga=\gamma

\allowdisplaybreaks

\begin{document}
\title[On an improved restricted reverse weak-type bound]
{On an improved restricted reverse weak-type bound for the maximal operator}

\author[A.K. Lerner]{Andrei K. Lerner}
\address[A.K. Lerner]{Department of Mathematics,
Bar-Ilan University, 5290002 Ramat Gan, Israel}
\email{lernera@math.biu.ac.il}

\thanks{The author was supported by ISF grant no. 1035/21.}

\begin{abstract}
We obtain an improved lower bound for the restricted reverse weak-type estimate of the Hardy–Littlewood maximal operator $M$. This result is applied to
the $\lambda$-median maximal operator $m_{\la}$ acting on a Banach function space $X$. We show that under certain assumptions on $X$, the boundedness
properties of $m_{\la}$ and $M$ are equivalent.
\end{abstract}


\maketitle

\section{Introduction}
Let $Q\subset {\mathbb R}^n$ be a cube. Denote by $M_Q$ the Hardy--Littlewood maximal operator restricted to $Q$, that is,
$$M_Qf(x):=\sup_{R\ni x, R\subset Q}\frac{1}{|R|}\int_R|f|,$$
where the supremum is taken over all cubes $R\subset Q$ containing the point $x\in Q$.

For $\la\in (0,1)$ define a lower Tauberian-type constant
$$C_n(\la):=\inf_{E\subset Q:0<|E|\le \la|Q|}\frac{|\{x\in Q:M_Q\chi_E(x)>\la\}|}{|E|}.$$

If the infimum above is replaced by a supremum, the resulting quantity is well known (and has been studied far beyond the cubic maximal operator). In that setting one obtains the so-called halo function, and there is an extensive literature on its behavior as $\la\to 1^-$; see, for example, \cite{HP14,HP15,S93}.

In contrast, it appears that the behavior of $C_n(\la)$ as $\la\to 1^-$ has not been investigated, even in the cubic case. The only classical lower bound (implicit in Stein \cite{S69}) comes from the Calder\'on--Zygmund decomposition and yields
$$C_n(\la)\ge \frac{1}{2^n\la}.$$
This estimate is meaningful only for $\la\in (0, 1/2^n)$. For $\la\ge 1/2^n$ it becomes trivial, since $M_Q\chi_E\ge\chi_E$ a.e. on $Q$,  and hence $C_n(\la)\ge 1$ for all $\la\in (0,1)$.
However, it is natural to expect that  $C_n(\la)>1$ for every $\la$ arbitrarily close to $1$.

Our first result is an improved lower bound for $C_n(\la)$ which is new for all $\la\in (c_n, 1)$, where $c_n<1/2^n$.

\begin{theorem}\label{mr} There exists a constant $B_n$ depending only on $n$ such that
$$C_n(\la)\ge 1+\frac{1-\la}{B_n\la}$$
for all $\la\in (0,1)$.
\end{theorem}

The proof of this result follows a method from the paper by Ivanishvili--Jaye--Nazarov \cite{IJN17} where an improved lower bound for the $L^p$ norm of the maximal operator was obtained.

Theorem \ref{mr} says that the ``halo operator'' $H_{\la}(E):=\{M_Q\chi_E>\la\}$ enlarges a set $E$ for every $\la\in (0,1)$. In particular, one can iterate $H^k_{\la}(E)$ until it saturates at $Q$. We will use this
iteration principle by applying it in a slightly different form to the $\la$-median maximal operator (also called the local maximal operator) defined by
$$m_{\la}f(x):=\sup_{Q\ni x}(f\chi_Q)^*(\la|Q|),$$
where $f^*$ denotes the non-increasing rearrangement, and the supremum is taken over all cubes containing the point $x$.

It follows from the definition that $m_{\la}f$ is non-increasing in $\la$. In fact, it is not difficult to show (see the Appendix) that $\lim_{\la\to 1^-}m_{\la}f=|f|$ almost everywhere.
Moreover, by Chebyshev's inequality, for every $\d>0$ we have $m_{\la}f\le \la^{-1/\d}M_{\d}f$, where $M_{\d}f:=M(|f|^{\d})^{1/\d}$, and $M$ denotes the maximal operator on ${\mathbb R}^n$.
These observations indicate that $m_{\la}$ is essentially smaller than $M$. However, our next results show that, under several additional assumptions, the boundedness of $m_{\la}$ on a Banach function space $X$ can be upgraded to that of $M$ on $X$.

For a Banach function space $X$ over ${\mathbb R}^n$ and $p>0$ set
$$\|f\|_{X^p}:=\||f|^{1/p}\|_{X}^{p}.$$

\begin{theorem}\label{itr} Suppose that $m_{\la}$ is bounded on $X$ for some $\la\in (0,1)$. Then there exists $p>1$ such that $M$ is bounded on $X^{1/p}$.
\end{theorem}

The main idea of the proof is iterating the boundedness of $m_{\la}$ on $X$. This leads to considering the iteration $m_{\la}(m_{\la}f)$. Theorem \ref{mr} plays the crucial role in order to handle this object for $\la$ close to $1$.

By $X'$ we denote the associated space of $X$. Then, using Theorem~\ref{itr} we obtain the following.

\begin{theorem}\label{char}
Let $X$ be a Banach function space. Assume that there exists $p_0>1$ such that, for all $p\ge p_0$,
boundedness of the Hardy--Littlewood maximal operator $M$ on $(X')^{1/p}$ implies its boundedness on $\big[(X')^{1/p}\big]'$.
If $m_{\la}$ is bounded on $X$ and $X'$ for some $\la\in (0,1)$, then $M$ is bounded on $X$.
\end{theorem}

At first sight, the assumption in Theorem~\ref{char} may appear rather strong. On the other hand, there is the following conjecture proposed by Nieraeth~\cite{N26}.

\begin{con}[{\cite{N26}}]\label{ncon}
Let $s\in (1,\infty)$, and let $X$ be an $s$-concave Banach function space. If $M$ is bounded on $X$, then $M$ is bounded on~$X'$.
\end{con}

For most standard Banach function spaces, Conjecture \ref{ncon} is true (see \cite{N26} for a discussion). However, in full generality it remains open.

Suppose now that $X$ is an $r$-convex Banach function space for some $r\in (1,\infty)$. Then $X'$ is $r'$-concave, and $(X')^{1/p}$ is $pr'$-concave.
Hence, assuming Conjecture~\ref{ncon}, the hypothesis of Theorem~\ref{char} is satisfied for all $p\ge p_0>1$. In particular, we obtain the following observation related to Conjecture~\ref{ncon}.

\begin{cor}\label{assum}
Assume Conjecture~\ref{ncon}. Let $X$ be an $r$-convex Banach function space for some $r\in (1,\infty)$.
If $m_{\la}$ is bounded on $X$ and $X'$ for some $\la\in (0,1)$, then $M$ is bounded on $X$.
\end{cor}

It would be interesting to find a direct proof of this result without appealing to Conjecture \ref{ncon}.
Alternatively, Corollary \ref{assum} suggests a possible strategy for disproving Conjecture~\ref{ncon}.

\section{Preliminaries}
Let $L^0({\mathbb R}^n)$ denote the space of measurable functions on ${\mathbb R}^n$. A vector space $X\subseteq L^0({\mathbb R}^n)$ equipped with a norm $\|\cdot\|_X$
is called a Banach function space over ${\mathbb R}^n$ if it satisfies the following properties:
\begin{itemize}
\item {\it Ideal property}: If $f\in X$ and $g\in L^0({\mathbb R}^n)$ with $|g|\le |f|$, then $g\in X$ and $\|g\|_X\le \|f\|_{X}$.
\item {\it Fatou property}: If $0\le f_j\uparrow f$ for $\{f_j\}$ in $X$ and $\sup_j\|f_j\|_{X}<\infty$, then $f\in X$ and $\|f\|_{X}=\sup_j\|f_j\|_{X}$.
\item {\it Saturation property}: For every measurable set $E\subset {\mathbb R}^n$ of positive measure, there exists a measurable subset $F\subseteq E$ of positive measure
such that $\chi_F\in X$.
\end{itemize}

We refer to a recent survey by Lorist and Nieraeth \cite{LN24} about (quasi)-Banach function spaces, where, in particular, one can find a discussion about the above choice of axioms.

Given a Banach function space $X$, we define the associate space (also called the K\"othe dual) $X'$ as the space of all $f\in L^0({\mathbb R}^n)$ such that
$$\|f\|_{X'}:=\sup_{\|g\|_{X}\le 1}\int_{{\mathbb R}^n}|fg|<\infty.$$
By the Lorentz--Luxembourg theorem (see \cite[Th. 71.1]{Z67}), we have $X''=X$ with equal norms.

Let $X$ be a Banach function space, and let $1\le p, q\le \infty$. We say that $X$ is $p$-convex if
$$\|(|f|^p+|g|^p)^{1/p}\|_{X}\le (\|f\|_{X}^p+\|g\|_{X}^p)^{1/p},\quad f,g\in X,$$
and we say that $X$ is $q$-concave if
$$(\|f\|_{X}^q+\|g\|_{X}^q)^{1/q}\le \|(|f|^q+|g|^q)^{1/q}\|_{X} ,\quad f,g\in X.$$

By a weight we mean a non-negative locally integrable function on~${\mathbb R}^n$. Given a weight $w$ and a cube $Q\subset {\mathbb R}^n$,
denote $\langle w\rangle_Q:=\frac{1}{|Q|}\int_Qw$.

Recall that a weight $w$ satisfies the $A_p, 1<p<\infty,$ condition if
$$[w]_{A_p}:=\sup_{Q}\langle w\rangle_Q\langle w^{-p'/p}\rangle_Q^{p/p'}<\infty.$$

Let $X$ be a Banach function space, and let $1<p<\infty$. We say that $X$ is $A_{p}$-regular if there exist $C_1, C_2>0$ such that for every $f\in X$ there is an $A_{p}$ weight $w\ge |f|$ a.e. with $[w]_{A_{p}}\le C_1$
and $\|w\|_{X}\le C_2\|f\|_{X}$.

The following results are due to Rutsky \cite[Th. 2]{R15}, \cite[Prop. 7]{R14} (see also \cite{L25} for alternative proofs).

\begin{theorem}[\cite{R15}]\label{rch} Let $X$ be a Banach function space, and let $1<p<\infty$. The following statements are equivalent:
\begin{enumerate}[(i)]
\item The maximal operator $M$ is bounded on $X^{1/p}$ and on $(X^{1/p})'$;
\item $X'$ is $A_p$-regular.
\end{enumerate}
\end{theorem}

\begin{theorem}[\cite{R14}]\label{propr} Let $X$ be a Banach function space such that $X$ is $A_p$-regular for some $1<p<\infty$. Suppose that there exists $\d>0$ such
that $M$ is bounded on $X^{\d}$. Then $M$ is bounded on $X$.
\end{theorem}

\section{Proof of Theorem \ref{mr}}
An important ingredient of the proof is the following covering lemma established by Mateu--Matilla--Nicolau--Orobitg \cite{MMNO00}.

\begin{lemma}\label{cover}
Let $E$ be a measurable subset of $Q$ with $|E|\le \la|Q|$, where $\la\in (0,1)$. Then there exists a sequence $\{R_j\}$
of cubes contained in $Q$ such that
\begin{enumerate}
\renewcommand{\labelenumi}{(\roman{enumi})}
\item
$|R_j\cap E|=\la|R_j|$;
\item
the family $\{R_j\}$ is almost disjoint with constant $B_n$, that is, every point of $Q$
belongs to at most $B_n$ cubes $R_j$;
\item
$E'\subset\cup_jR_j$, where $E'$ is the set of the density
points of $E$.
\end{enumerate}
\end{lemma}

Observe that the constant $B_n$ in Lemma \ref{cover} comes from the Besicovitch covering theorem.

\begin{proof}[Proof of Theorem \ref{mr}] We use the same method as in \cite{IJN17}.
Let us apply Lemma \ref{cover}, and let $\{R_{\la,j}\}$ be the corresponding cubes. Define
$$\psi(x,\la):=\sum_j\chi_{R_{\la,j}}(x).$$
Then summing up $|R_{\la,j}\cap E|=\la|R_{\la,j}|$ yields
$$
\int_E\psi(x,\la)dx=\la\int_Q\psi(x,\la)dx.
$$

If $M_Q\chi_E(x)<\la$, then no $R_{\la,j}$ contains $x$, and hence $\psi(x,\la)=0$. Therefore, the above equality implies
\begin{equation}\label{eq}
(1-\la)\int_E\psi(x,\la)dx=\la\int_{\{x\in Q:M_Q\chi_E\ge \la\}\setminus E}\psi(x,\la)dx.
\end{equation}

By property (ii) of Lemma \ref{cover}, $\psi(x,\la)\le B_n$. Next, by property (iii), $\psi(x,\la)\ge 1$ for all $x\in E'$. Since $|E'|=|E|$, by (\ref{eq}) we obtain
$$(1-\la)|E|\le B_n\la\big(|\{x\in Q:M_Q\chi_E(x)\ge \la\}|-|E|\big).$$
Hence,
$$|\{x\in Q:M_Q\chi_E(x)\ge \la\}|\ge \Big(1+\frac{1-\la}{B_n\la}\Big)|E|,$$
which implies the same estimate for $|\{x\in Q:M_Q\chi_E(x)>\la\}|$, and the proof is complete.
\end{proof}

\section{Proof of Theorems \ref{itr} and \ref{char}}
Recall that, by the definition of the non-increasing rearrangement,
$$(f\chi_Q)^*(\la|Q|)=\inf\{\a>0:|\{x\in Q:|f(x)|>\a\}|\le \la|Q|\},$$
from which, for every $\a>0$,
$$(f\chi_Q)^*(\la|Q|)>\a\Leftrightarrow |\{x\in Q:|f(x)|>\a\}|>\la|Q|.$$
Hence,
\begin{equation}\label{ident}
\{x:m_{\la}f(x)>\a\}=\{x:M\chi_{\{|f|>\a\}}(x)>\la\}.
\end{equation}

\begin{lemma}\label{comp} Let $0<\la,\eta<1$. There exists a constant $\ga=\ga(n,\la)<1$ such that for all $x\in {\mathbb R}^n$,
\begin{equation}\label{adv}
m_{\ga\eta}f(x)\le m_{\eta}(m_{\la}f)(x).
\end{equation}
\end{lemma}

\begin{proof} By Theorem \ref{mr} along with (\ref{ident}),
$$|\{x\in Q:m_{\la}f(x)>\a\}|\ge C(n,\la)|\{x\in Q:|f(x)|>\a\}|,$$
where $C(n,\la):=1+\frac{1-\la}{B_n\la}$. Hence, for $\ga:=\frac{1}{C(n,\la)}$,
$$(f\chi_Q)^*(\ga\eta|Q|)\le \big((m_{\la}f)\chi_Q\big)^*\big(\eta|Q|\big),$$
from which the desired estimate follows.
\end{proof}

\begin{remark}\label{rem} A similar estimate was obtained in \cite{LP07} but using that $C_n(\la)\ge \frac{1}{2^n\la}$ instead of the bound from Theorem \ref{mr}. Namely, it was shown in \cite{LP07} that
$$m_{2^n\la\eta}f(x)\le m_{\eta}(m_{\la}f)(x).$$
This estimate holds under the restriction $\la\eta<1/2^n$. The advantage of~(\ref{adv}) is that it holds for all $\la,\eta\in (0,1)$. This will be used in the proof of Theorem \ref{itr}.
\end{remark}

\begin{prop}\label{maxf}
Let $X$ be a Banach function space. Then
$$\|Mf\|_{X}\le \int_0^1\|m_{\la}f\|_{X}d\la.$$
\end{prop}

\begin{proof} Using that
$$\frac{1}{|Q|}\int_Q|f|=\int_0^1(f\chi_Q)^*(\la|Q|)d\la,$$
we obtain
$$Mf(x)\le \int_0^1m_{\la}f(x)d\la.$$
Therefore, for every $g$ with $\|g\|_{X'}=1$,
$$\int_{{\mathbb R}^n}(Mf)g\,dx\le \int_0^1\int_{{\mathbb R}^n}(m_{\la}f)g\,dx\,d\la\le \int_0^1\|m_{\la}f\|_{X}d\la,$$
which, by duality (along with the Lorentz--Luxembourg theorem), completes the proof.
\end{proof}

\begin{proof}[Proof of Theorem \ref{itr}] By the assumption, there exist $\la_0\in (0,1)$ and $C>0$ such that for every $f\in X$,
$$\|m_{\la_0}f\|_{X}\le C\|f\|_{X}.$$
From this, by Lemma \ref{comp}, there exists $\ga=\ga(n,\la_0)<1$ such that
$$\|m_{\ga\la_0}f\|_{X}\le \|m_{\la_0}(m_{\la_0}f)\|_{X}\le C^2\|f\|_{X}.$$
By the same argument,
$$\|m_{\ga^2\la_0}f\|_{X}\le \|m_{\ga\la_0}(m_{\la_0})\|_{X}\le C^3\|f\|_{X}.$$
By induction we obtain that for all $k=0,1,\dots,$
$$\|m_{\ga^k\la_0}f\|_{X}\le  C^{k+1}\|f\|_{X}.$$

Assume now that $\la\in (0,\la_0]$. Take $k$ such that $\la\in [\ga^{k+1}\la_0,\ga^k\la_0]$. Setting $r:=\frac{\log C}{\log(1/\ga)}$ and using that $m_{\la}$ is non-increasing in $\la$, we obtain
$$\|m_{\la}f\|_{X}\le \|m_{\ga^{k+1}\la_0}f\|_{X}\le C^{k+2}\|f\|_{X}\le C^2\la_0^r\frac{1}{\la^r}\|f\|_{X}.$$
If $\la\in (\la_0,1)$, then trivially, by the previous estimate,
$$\|m_{\la}f\|_{X}\le \|m_{\la_0}f\|_{X}\le C^2\|f\|_{X}.$$
Therefore, for all $\la\in (0,1)$,
$$\|m_{\la}f\|_{X}\le C^2\frac{1}{\la^r}\|f\|_{X}.$$

From this, for $p>0$, using that $(m_{\la}f)^s=m_{\la}(|f|^s)$ for every $s>0$, we obtain
$$\|m_{\la}f\|_{X^{1/p}}=\|m_{\la}(|f|^{p})\|_{X}^{1/p}\le \Big(C^2\frac{1}{\la^r}\Big)^{1/p}\|f\|_{X^{1/p}}.$$
Taking now any $p>0$ such that $p>max(1,r)$ completes the proof by Proposition \ref{maxf}.
\end{proof}

\begin{proof}[Proof of Theorem \ref{char}]
If $m_{\la}$ is bounded on $X$ and on $X'$ for some $\la\in (0,1)$, then, by Theorem \ref{itr}, there exist $r,q>1$ such that $M$ is bounded on $X^{1/r}$ and on $(X')^{1/q}$.

By H\"older's inequality, $M$ is bounded on $(X')^{1/p}$ for all $p\ge q$. Therefore, by the hypothesis, $M$ is bounded on $\big[(X')^{1/p}\big]'$ for $p\ge max(q,p_0)$.
From this, by Theorem \ref{rch} applied to $X'$ instead of $X$, we obtain that $X$ is $A_p$-regular for $p=\max(q,p_0)$. Using that $M$ is bounded on $X^{1/r}$, we obtain, by Theorem \ref{propr},
that $M$ is bounded on $X$.
\end{proof}

In conclusion we observe that the following version of Theorem \ref{char} holds with essentially the same proof.

\begin{theorem}\label{char1}
Let $X$ be a Banach function space with the following property: for every $p\ge 1$, the Hardy--Littlewood maximal operator $M$ is bounded on $X^{1/p}$ if and only if $M$ is bounded on $(X^{1/p})'$. Then the following are equivalent:
\begin{enumerate}[(i)]
\item $M$ is bounded on $X$ and on $X'$;
\item $m_{\la}$ is bounded on $X$ and on $X'$ for some $\la\in (0,1)$.
\end{enumerate}
\end{theorem}

Indeed, the implication $(i)\Rightarrow (ii)$ is trivial because, by Chebyshev's inequality, $m_{\la}f\le \frac{1}{\la}Mf$. To show that $(ii)\Rightarrow (i)$, we use the same arguments as in the proof of Theorem \ref{char}.
Namely, we obtain that $M$ is bounded on $X^{1/r}$ and on $(X')^{1/q}$. Hence, $M$ is bounded on $(X^{1/r})'$, and so $X'$ is $A_r$-regular. This, along with boundedness on $(X')^{1/q}$, implies that $M$ is bounded on $X'$.
From this, by the hypothesis with $p=1$, we have that $M$ is bounded on $X$.

\section{Appendix}
Denote by $S_0({\mathbb R}^n)$ the set of all measurable functions $f$ on ${\mathbb R}^n$ such that
$$|\{x\in {\mathbb R}^n:|f(x)|>\a\}|<\infty$$
for all $\a>0$.

We will prove the following statement.

\begin{prop}\label{prml} For all $f\in S_0({\mathbb R}^n)$,
$$\lim_{\la\to 1^-}m_{\la}f(x)=|f(x)|$$
almost everywhere.
\end{prop}

This proposition is not used in the proofs of the main results of the paper. However, it answers a question implicit in those results: what happens to $m_{\la}$ as $\la$ approaches 1?

\begin{proof}[Proof of Proposition \ref{prml}] Denote $\bar f(x):=\lim_{\la\to 1^-}m_{\la}f(x)$. It is well known (see, e.g., \cite[Lemma 6]{L04}) that
$$|f(x)|\le m_{\la}f(x)$$
for every point of approximate continuity of $f$, for all $\la\in (0,1)$. Therefore, $|f(x)|\le \bar f(x)$ almost everywhere.

To show the reverse estimate, define the halo function of the maximal operator by
$$\f(\la):=\sup_{E}\frac{|\{x\in {\mathbb R}^n:M\chi_{E}(x)>\la\}|}{|E|},\quad\la\in (0,1),$$
where the supremum is taken over all measurable sets of positive finite measure.
It was shown by Solyanik \cite{S93} that $\lim_{\la\to 1^-}\f(\la)=1$.

Denote $E:=\{x\in {\mathbb R}^n: \bar f(x)>|f(x)|\}$. Then, for non-negative rational numbers $r_k$, setting
$$E_k:=\{x: \bar f(x)>r_k\}\setminus \{x:|f(x)|>r_k\},$$
we have $E=\cup_kE_k$. Next, for all $\la\in (0,1)$,
$$E_k\subset \{x: m_{\la}f(x)>r_k\}\setminus \{x:|f(x)|>r_k\}.$$
Therefore, by (\ref{ident}),
$$|E_k|\le (\f(\la)-1)|\{x:|f(x)|>r_k\}|.$$
Hence, letting $\la\to 1^-$ yields $|E_k|=0$, and thus $|E|=0$, which completes the proof.
\end{proof}

\end{document}